\documentclass[twoside,12pt, leqno]{article}
\usepackage{amsmath,amscd,amsthm,amssymb,amsxtra,latexsym,epsfig,epic,graphics}
\usepackage[matrix,arrow,curve]{xy}
\usepackage{graphicx}
\usepackage[breaklinks,bookmarksopen,bookmarksnumbered]{hyperref}
\hypersetup{colorlinks=true,backref=true,citecolor=blue}

\def\antiddot{\mathinner{\mkern1mu\raise1pt\vbox{\kern7pt\hbox{.}}\mkern2mu
        \raise4pt\hbox{.}\mkern2mu\raise7pt\hbox{.}\mkern1mu}}

\newcommand{\FF}{{\mathbb F}}

\newcommand{\PP}{{\mathbb P}}
\newcommand{\QQ}{{\mathbb Q}}

\newcommand{\ZZ}{{\mathbb Z}}

\newcommand{\Ext}{{\rm{Ext}}}

\newcommand{\coker}{{\rm{coker}\,}}

\newcommand{\s}{\mathcal}

\newcommand{\sE}{{\s E}}
\newcommand{\sF}{{\s F}}
\newcommand{\sG}{{\s G}}
\newcommand{\sH}{{\s H}}
\newcommand{\sI}{{\s I}}

\newcommand{\sL}{{\s L}}
\newcommand{\sM}{{\s M}}

\newcommand{\sO}{{\s O}}

\newcommand{\lto}{\leftarrow}
\newcommand{\tensor}{\otimes}

\newcommand{\punkt}{\hspace{-.3ex}\raise.15ex\hbox to1ex{\Huge.}}

\DeclareMathOperator{\Pic}{Pic}

\DeclareMathOperator{\Hilb}{Hilb}

\DeclareMathOperator{\Spec}{Spec}

\DeclareMathOperator{\Hom}{Hom}

\DeclareMathOperator{\image}{image}

\DeclareMathOperator{\pd}{pd}

\DeclareMathOperator{\PGL}{PGL}

\DeclareMathOperator{\codim}{codim}
\DeclareMathOperator{\rank}{rank}

\newcommand{\gm}{\mathfrak m}

\newcommand{\Mac}{{\texttt {Macaulay2}}}
\newtheorem{theorem}{Theorem}[section]

\newtheorem{proposition}[theorem]{Proposition}
\newtheorem{corollary}[theorem]{Corollary}
\newtheorem{conjecture}[theorem]{Conjecture}
\theoremstyle{definition}

\newtheorem{remark}[theorem]{Remark}

\newtheorem{example}[theorem]{Example}

\makeatletter
\def\Ddots{\mathinner{\mkern1mu\raise\p@
\vbox{\kern7\p@\hbox{.}}\mkern2mu
\raise4\p@\hbox{.}\mkern2mu\raise7\p@\hbox{.}\mkern1mu}}
\makeatother

\usepackage{times}
\newdimen\x \x=12pt

\usepackage{color}

\date{}
\title{Matrix factorizations and families of curves of genus 15}
\author{Frank-Olaf Schreyer
\footnote{This paper reports on work done during the Commutative Algebra Program, 2012-13,
at MSRI. I am grateful
to MSRI for financial support and for providing such an exciting environment.}}

\begin{document}

\maketitle

\begin{abstract}
In this note, we explain how certain matrix factorizations on cubic threefolds lead to families of curves of genus $15$ and degree $16$ in $\PP^4$. We prove that the moduli space
$\widetilde \sM_{15,16}^4 =\{(C,L)\mid C \in M_{15}, \ L \in W^4_{16}(C) \subset \Pic^{16}(C)\}$ is uniruled, and that  $\widetilde \sM_{15,16}^4$ is birational to a space of certain matrix factorizations on cubics.
 Our attempt to prove the unirationality of this space failed with our methods. 
Instead one can interpret our findings as evidence for the conjecture that the basis of the maximal rational connected fibration of  $\widetilde \sM_{15,16}^4$ has a three dimensional base.
\end{abstract}

\section*{Introduction} The moduli spaces $\sM_g$ of curves of genus $g$ are known to be unirational for $g\le 14$, \cite{Sev,Ser1,CR1,Verra}. For $g=22$ or $g\ge 24$ they are known to be of general type \cite{HM,EH,F,F2}. The cases in between are not fully understood: $\sM_{23}$ has positive Kodaira dimension \cite{F}, $\sM_{15}$ is rationally connected \cite{CR2,BV}, and  $\sM_{16}$ \cite{CR3, F2} is uniruled.
In this paper we are mainly concerned with $\sM_{15}$ and an  attempt to prove its unirationality. 

By Brill-Noether theory, a general curve of genus $15$ has a smooth model of degree $16$ in $\PP^4$.
Let
$$
\sH \subset \Hilb_{16t+1-15}(\PP^4)
$$
be the component of the Hilbert scheme of curves of degree $d=16$ and genus $g=15$ in $\PP^4$, which dominates 
 the moduli space $\sM_{15}$. (This component is unique because the Brill-Noether dual models form  the Severi variety of plane curves of degree $d=12$, geometric genus $g=15$  
and $\delta=40$ nodes, which is known to be irreducible \cite{H}.)
Let
$$ 
\widetilde \sM^4_{15,16} \subset \{ (C,L) \mid C \in \sM_{15}, L \in W^4_{16}(C) \} 
$$
be the component which dominates $\sM_{15}$. 
So $\sH// PGL(5)$ is birational to $\widetilde \sM^4_{15,16}$. Our main result connects this moduli space to a moduli space of certain matrix factorizations.

\begin{theorem}\label{main}
The moduli space $\widetilde \sM_{15,16}^4$ of curves of genus $15$ together with a $g^4_{16}$ is birational to a component of the moduli space
of matrix factorizations of type 
$
( \psi\colon \sO^{18}(-3)\to \sO^{15}(-1)\oplus \sO^3(-2), \varphi\colon\sO^{15}(-1)\oplus \sO^3(-2) \to \sO^{18} )
$ 
of cubic forms on $\PP^4$.
\end{theorem}

As a corollary of our proof we obtain the dimension statement in

\begin{theorem}\label{OnGenCubic}
A general cubic threefold in $\PP^4$ contains a $32$-dimensional uni\-ruled family of smooth curves of genus $15$ and degree $16$.
\end{theorem}

Since a general curve in $\sH$ lies on a unique cubic threefold, and cubic threefolds depend on $10$ parameters up to projectivities, the dimension $32$ fits with 
$\dim \sM_{15} = 42$. 

Our approach to construct a family of curves of genus $15$ builds upon the construction of a matrix factorization on  a cubic as  a syzygy module of an auxiliary module $N$.
We use Boij-S\"oderberg theory \cite{BS1}, \cite{ES}, \cite{BS2}, \cite{SE} and the \Mac \ package \cite{ESS} to get a list of candidate Betti tables.
In all our cases the sheaf $\sL=\widetilde N$ will be a line bundle on an auxiliary curve $E$. The choice of $E$ and $\sL$ is motivated by a dimension count and the shape of the 
Betti table of $N$. We succeeded to construct altogether 20 families of curves  in $\sH$, and 17 of the families are unirational. However, the unirational families  do not dominate $\sM_{15}$
although the number of parameters in the construction exceeds $42$. Three of these families have  a non-unirational  step in their construction. (We need an effective divisor on the auxiliary curve). Precisely, those three families dominate $\sM_{15}$. 
We use the family from Theorem \ref{fam4} to prove 

\begin{theorem} \label{uniruled}
The moduli space $\widetilde \sM^4_{15,16}$ is uniruled.
\end{theorem}

\noindent and the uniruledness in Theorem \ref{OnGenCubic}. Furthermore we get

\begin{theorem}\label{random curve} There exists a probabilistic algorithm which randomly produces  curves of genus $g=15$ over a finite field $\FF_q$ with $q$ elements from a Zariski open subset of $\sM_{15}$ in  running time $O((\log q)^3)$.
\end{theorem}

\noindent
The existence of such an algorithm in principle, is no surprise. Important is that the algorithm actually runs in reasonable time on current computer algebra systems. 

The proofs of the Theorems in this article rely on computer algebra. An implementation  of all necessary computations can be found in the \Mac \  package 
\href{http://www.math.uni-sb.de/ag/schreyer/home/computeralgebra.htm}{MatFac15} available online.

Many of the images of the unirational families  have dimension $39$. There is one of dimension $41$, one of dimension $40$,  and  some of dimension $< 39$. A good explanation 
why I failed to prove the unirationality of $\sM_{15}$ with this method could be

\begin{conjecture} \label{mrc}
The maximal rationally connected fibration   of $\widetilde \sM_{15,16}^4$ has a three dimensional base. 
\end{conjecture}

\section{Matrix factorizations}\label{MatFac}

Matrix factorizations were introduced 1980 by David Eisenbud in his seminal paper  \cite{E1}. We recall basic facts. Let $R$ be a regular local ring and $f \in R$ not a unit.
A \textit{matrix factorization} of $f$ is a pair $(\varphi,\psi)$ of matrices satisfying  $\psi\circ \varphi= f id$ and $\varphi\circ \psi = f id$ .Then $\varphi, \psi$ are necessarily
square matrices of the same size. If $(\varphi,\psi)$ is  a matrix factorization, then $\coker \varphi$ is a maximal Cohen-Macaulay module (MCM) on the hypersurface ring $R/f$.
Conversely, given a finitely generated maximal Cohen-Macaulay module $M$ over $R/f$, it has a short minimal free resolution
$$
0 \longleftarrow M \longleftarrow F\longleftarrow  G \longleftarrow 0
$$
as an $R$-module, and multiplication with $f$ on this complex is null homotopic
$$
\xymatrix{ 
0 & \ar[l] M \ar[d]_0&  \ar[l] F \ar[d]_f \ar[dr]^\psi& \ar[l]_\varphi G \ar[d]^f &  \ar[l] 0 \\
0 & \ar[l] M &  \ar[l] F  & \ar[l]^\varphi G &  \ar[l] 0 \\
}
$$
which yields a matrix factorization $(\varphi,\psi)$. As an $R/f$-module, $M$ has the infinite 2-periodic resolution
$$
\xymatrix{
0 &\ar[l] M & \ar[l] \overline F&  \ar[l]_{\overline \varphi}  \overline G & \ar[l]_{\overline \psi} \overline F &\ar[l]_{\overline \varphi}  \overline G & \ar[l]_{\overline \psi} \ldots \\
}
$$
where $\overline F=F \tensor R/f$ and $\overline G=G \tensor R/f$. In particular, this sequence is exact, and the dual sequence corresponding to the matrix factorization
$(\psi^t,\varphi^t)$ is exact as well.

If $N$ is an arbitrary $R/f$ module, then the minimal free resolution becomes eventually 2-periodic:
If 
$$
 0 \longleftarrow N \longleftarrow F_0\longleftarrow  F_1 \longleftarrow \ldots \longleftarrow  F_c \longleftarrow 0
$$
is the minimal free resolution of $N$ as an $R$-module, then a (not necessarily minimal) free resolution of $N$ as $R/f$-module starts
$$
 0 \leftarrow N \leftarrow \overline F_0\leftarrow  \overline F_1 \leftarrow  \overline F_2 \oplus \overline F_0\leftarrow \overline F_3\oplus \overline F_1 \leftarrow  
 \ldots \leftarrow \overline F_{ev} \leftarrow \overline F_{odd} \leftarrow  \ldots
$$
where 
$$
F_{ev} =\bigoplus_{i \equiv 0  \mod 2} F_i \quad \hbox{ and  }\quad F_{odd} =\bigoplus_{i \equiv 1 \mod 2} F_i .
$$ 
The high syzygy modules over a Cohen-Macaulay ring are MCM. In case of an hypersurface $M =\coker(\overline F_{odd} \to \overline F_{ev})$ is a MCM module.
There is  a natural surjection from $M\oplus F$ to $N$ with kernel $P$,
$$ 
0 \leftarrow N \leftarrow M\oplus F  \leftarrow P \leftarrow 0
$$
where $F$ is a free $R/f$-module and $P$ is a module of  finite projective dimension. In the examples relevant  later on, we will find that we can choose $F=0$.

Thus, an arbitrary $R/f$-module can be build from an MCM-module and a module of finite projective dimension.
In a remarkable paper \cite{AB} of Auslander and Buchweitz on Maximal Cohen-Macaulay approximation this phenomenon is studied in much wider generality.

In the case of interest for this paper, we replace $R$ by the standard graded polynomial ring $S$, the homogeneous coordinate ring of some $\PP^n$, and $f$ by a homogeneous form of degree $d$. We have to take the grading into account. A matrix factorization is a pair
$(\varphi: G \to F, \psi: F\to G(d))$ 
where $F=\bigoplus_{\ell=1}^r S(-a_\ell),G=\bigoplus_{\ell=1}^r S(-b_\ell)$. 
If $N$
is an $S/f$-module  with minimal free resolution $F_\bullet$ as an $S$-module then the free resolution as an $S/f$- module has as $i$-th term the module
$
\overline F_i \oplus \overline F_{i-2}(-d) \oplus \ldots \oplus  \overline F_0(-id/2) \hbox{ or } \overline F_i \oplus \overline F_{i-2}(-d)\oplus \ldots \oplus  \overline F_1(-(i-1)d/2)
$
in case $i$ is even or odd, respectively.

The associated sheaf $\sF=\widetilde M$ of $M=\coker(F \to G)$ is a sheaf of maximal  Cohen-Macaulay modules on the scheme $X \subset \PP^n$ defined by $f$. Thus, if $X$ is smooth then  $\sF$ is locally free, i.e., a vector bundle on $X$. In this case  $\det \varphi =\lambda f^k$ for a unit $\lambda \in K \subset S$ and $\rank \sF = k$. We frequently use the sheafified notation
$$(
 \varphi: \bigoplus_{\ell=1}^r \sO(-b_\ell) \to \bigoplus_{\ell=1}^r \sO(-a_\ell),\psi: \bigoplus_{\ell=1}^r \sO(-a_\ell) \to \bigoplus_{\ell=1}^r \sO(d-b_\ell) )
 $$
for matrix factorizations. From the short exact sequence
$$
\xymatrix{ 
0  \ar[r] & \bigoplus_{\ell=1}^r \sO(-b_\ell) \ar[r]^\varphi & \bigoplus_{\ell=1}^r \sO(-a_\ell) \ar[r] & \sF  \ar[r] & 0 \\
}$$
we obtain that $\sF$ has no middle cohomology:
$$
H^i(X,\sF(j))= 0  \hbox{ for all $i$ with } 1 \le i \le \dim X -1 \hbox{ and  all } j \in \ZZ,
$$
i.e., $\sF$ is an arithmetically Cohen-Macaulay (ACM) bundle on $X$.
Conversely, if $\sF$ is an ACM-bundle on a (smooth) hypersurface $X \subset \PP^n$ then
$$
M=H^0_*(\sF) = \sum_{j \in \ZZ} H^0(\sF(j))
$$
is a MCM-module over $S/f$ where $\langle f \rangle$ is the homogeneous ideal of $X$. The investigation of ACM bundles on hypersurfaces is a widely studied subject which fairly recently caught the attention even of physicists. 

\section{Syzygies of the general curve in $\sH$}

Recall that $\sH$ denotes the component of the Hilbert scheme of curves of degree $d=16$ and genus $g=15$ in $\PP^4$ which dominate $\sM_{15}$. 

\begin{proposition} Let $C \in \sH$ be a general point.   The homogeneous coordinate ring $S_C=S/I_C$ and the section ring $\Gamma_*(\sO_C)=\oplus_{n\in \ZZ} H^0(\sO_C(n))$ have minimal free resolutions with the following Betti tables
\begin{center}
\begin{tabular}{c|c c c c c }
             &0 & 1& 2 & 3 & 4\\ \hline
          0& 1& .& . & . & .\\
          1& . &. & . & .& .\\
          2& . & 1 & . & . &  .\\
          3& .& 15 & 30& 18 & 3\\
\end{tabular} $\qquad$ and  $\qquad$
\begin{tabular}{c|c c c c }
             &0 & 1 &  2 & 3 \\ \hline
          0& 1  &. &  . & . \\
          1& . & .&  .& . \\
          2& 3 & 16 &  15 & . \\
          3& .  & . & . & 3 \\
\end{tabular}
\end{center}
respectively. In particular $C \subset \PP^4$ lies on a unique smooth cubic threefold $X$. The minimal resolution of $\Gamma_*(\sO_C)$ as a module over the homogeneous coordinate ring 
of $X$ is eventually 2-periodic with Betti numbers
\begin{center}
\begin{tabular}{c|c c c c c c c}
             &0 & 1 &  2 & 3 &4 &  $\cdots$ \\ \hline
          0& 1  &. &  . & . & . & .\\
          1& . & .&  .& . & . &. \\
          2& 3 & 15 &  15 & . \\
          3& .  & .    &      3& 18 & 15 & . \\
          $\vdots$& .  & .    & .       &   .  &   3& 18 \\ 
\end{tabular} 
\end{center}
\end{proposition}

\begin{proof} Assuming that the maps $H^0(\sO_{\PP^4}(n)) \to H^0(\sO_C(n))$ are of maximal rank for all $n$, i.e., $C$ has maximal rank, we find, using Riemann-Roch and the fact that
$\sO_C(n)$ is non-special for $n \ge 2$, that
\begin{itemize}
\item the Hilbert series of the homogeneous coordinate ring of $C$ is 
$$
H_C(t) =1 +5t+15t^2+34t^3+(34+16)t^4+(34+2\cdot 16)t^5 + \ldots,
$$
\item the Hartshorne-Rao module
$$H_*^1(\sI_C) = \sum_{n \in \ZZ}H^1(\sI_C(n)) \cong K^3(-2)$$
is a three dimensional vector space concentrated in degree $2$, 
\item the ideal sheaf $\sI_C$ is 4-regular, and
\item the homogeneous ideal $I_C = H^0_* ( \sI_C)$
has a single generator in degree 3 and 15 further generators in degree 4.
\end{itemize}
Hence, the Hilbert numerator has shape
$$
(1-t)^5H_C(t) = 1-t^3-15t^4+30t^5-18t^6+3t^7 ,
$$
and smooth maximal rank curves in $\sH$ have a Betti table as  claimed in the Proposition.  To establish that a general point in $C \in \sH$ is a maximal rank curve, it suffices to produce a single maximal rank example. We will explicitely construct such  examples in Section \ref{constructions}.  Moreover, by inspection we find that the general $C$ lies on a smooth cubic hypersurface $X$.

The Betti table of the resolution $F_\bullet$ of the section ring as an $S$-module can be deduced with the same method,
since $H_{\Gamma_*(\sO_C)}(t)=H_C(t)+3t^2$. The only questionable entry of the Betti table is $\beta_{ 2,5}^S(\Gamma_*(\sO_C))=0$ for which we argue as follows. The complex 
$\Hom(F_\bullet,S(-5))$ resolves $H^0_*( \sE xt^3(\sO_C,\omega_{\PP^4})) \cong H^0_*(\omega_C)$. By Brill-Noether theory,  there are no linear relations among the three generators in $H^0(\omega_C(-1))$ iff $(C,\sO_C(1))$ does not correspond to a ramification point of the map $\widetilde \sM_{15,16}^4 \to \sM_{15}$. So $\beta_{ 2,5}^S(\Gamma_*(\sO_C))=0$
holds for $C \in \sH$ outside the ramification divisor. 
 Finally,  we compute  the Betti number of  $\Gamma_*(\sO_C)$ as an $S_X$-module. The (possibly  non-minimal) resolution from Section \ref{MatFac} has the Betti table
\begin{center}
\begin{tabular}{c|c c c c c c c}
             &0 & 1 &  2 & 3 &4 &  $\cdots$ \\ \hline
          0& 1  &. &  . & . & . & .\\
          1& .   & .   &   1& . & . &. \\
          2& 3 & 16 &  15 & . & 1& .\\
          3& .  & .    &      3& 19 & 15 & . \\
          $\vdots$& .  & .    & .       &   .  &   3& 19 \\ 
\end{tabular} 
\end{center}
So this resolution is non-minimal. The minimal version has the desired Betti table.
\end{proof}

Consider the matrix factorization 
$$
(\varphi\colon\sO^{15}(-4)\oplus \sO^3(-5) \to \sO^{18}(-3), \psi\colon \sO^{18}(-3)\to \sO^{15}(-1)\oplus \sO^3(-2))
$$ 
corresponding (up to twist) to the 2-periodic part
of the resolution of  $\Gamma_*(\sO_C)$ as $S_X$-module. Let $\sF=\coker \varphi$.  The sheaf $\sF$ is a vector bundle of 
$\rank \sF =7$ on $X$, since $\deg \det \varphi= 15+2\cdot 3=3\cdot 7$. We have a short exact sequence
$$
\xymatrix{
0 & \ar[l] \coker \psi & \ar[l] \sO_X^{15}(-1)\oplus \sO_X^3(-2) & \ar[l] \sF &\ar[l] 0 \\
}.
$$
The composition $\sO_X^3(-2) \lto \sF \lto \sO_X^{18}(-3)$ is surjective with a summand $\sO_X^3(-3) $ in the kernel. Indeed, the composition has as a component the sheafified presentation matrix $S^3(-2) \lto S^{15}(-3)$ of the Hartshorne-Rao module  of $C$ restricted to $X$ as a summand, and surjectivity follows because 
$$\coker( S^{15}(-3) \to S^3(-2)) \cong K^3(-2)$$
is a module of finite length. Thus, we obtain a complex
$$
\xymatrix{
0  \ar[r] & \sO_X^3(-3)  \ar[r]^{\quad \beta} &  \sF \ar[r]^{\alpha\quad} &\ar[r]  \sO_X^3(-2) & 0 \\
}.
$$

\begin{theorem}\label{monad} Let $C \in \sH$ be a general point. Then the complex
$$
\xymatrix{
0  \ar[r] & \sO_X^3(-3)  \ar[r]^{\quad \beta} &  \sF \ar[r]^{\alpha\quad} &\ar[r]  \sO_X^3(-2) & 0 \\
}.
$$
obtained from the resolution of $\Gamma_*(\sO_C)$ as an $S_X$-module is a monad for the ideal sheaf $\sI_{C/X}$ of $C$ in $X$, i.e.,
$\alpha$ is surjective, $\beta$ is  injective and $\ker \alpha  /  \image \beta \cong \sI_{C/X}$.
\end{theorem}

\begin{proof} We already proved the surjectivity of $\alpha$. Thus $\sG = \ker \alpha$ is a rank 4 subbundle of $\sF$, and $\beta$ induces a homomorphism
$\beta'\colon \sO_X(-3)^3\to  \sG$ between locally free sheaves on $X$. We expect that $\beta'$ drops rank along a codimension 2 subscheme of $X$. This is the case if
$\coker(\beta^*\colon \sG^* \to \sO_X(3)) $ has support in codimension $2$ on $X$. The composition $\sO_X^{15}(1) \to  \sF^* \to \sG^* \to \sO_X^3(3)$ coincides with
the restriction to $X$ of  $\sH om_{\sO_{\PP^4}}(\sO^{3}(-6) \to \sO^{15}(-4), \omega_{\PP^4})$ up to twist, where 
$\sO^{15}(-4) \lto \sO^3(-6)$ denotes the sheafified last map in the resolution of $\Gamma_*(\sO_C)$ as an $S$-module.  Hence
  $\coker(\beta^*(-2):\sG^*(-2) \to \sO_X^3(1)) \cong \omega_C $ has support on $C$, which has codimension $2$ on $X$. Finally, $\omega_X \cong \sO_X(-2)$ implies 
  $\ker(\beta^*\colon \sG^* \to \sO_X^3(3)) \cong \sO_X$ and $ \ker \alpha / \image \beta \cong \coker(\beta') \cong \sI_{C/X} \subset \sO_X$.
\end{proof}

\noindent
{\it Proof} of Theorem \ref{main}. One direction follows from Theorem \ref{monad}. For the other direction, consider an arbitrary matrix factorization of
type 
$$
( \psi\colon \sO^{18}(-3)\to \sO^{15}(-1)\oplus \sO^3(-2), \varphi\colon\sO^{15}(-1)\oplus \sO^3(-2) \to \sO^{18} )
$$
of some cubic form as in Theorem \ref{main}. The pair $(\varphi(-3),\psi)$ is a matrix factorization of the shape used to derive the monad of Theorem \ref{monad}. In particular, we have again a short exact sequence
$$
 0 \to \sF \to \sO_X^{15}(-1)\oplus \sO_X^3(-2) \to \coker \psi \to 0
$$
with $\sF= \coker \varphi(-3)$. The composition $\sO_X^{18}(-3) \to \sF \to \sO_X^3(-2)$ has a summand $\sO_X^3(-3)$ in the kernel simply because
there are only five linearly independent linear forms on $\PP^4$. Thus we can derive a complex 
$$
0 \to \sO_X^3(-3) \to \sF \to \sO_X^3(-2) \to 0
$$
again.  It is an open condition on the matrices $\psi$ that the summand $\sO_X^3(-3)$ in the kernel is uniquely determined, and that $\sF \to \sO_X^3(-2)$ is surjective.
Further open conditions on matrix factorizations are the conditions that the complex above is a monad,  that it homology is the ideal sheaf $\sI_{C/X}$ of a smooth curve of degree $d=16$ and genus $g=15$ on a smooth cubic threefold $X$, that $C \subset \PP^4$ is a maximal rank curve, and that the pair $(C, \sO_C(1))$
does not lie in the ramification divisor of $\widetilde \sM_{15,16}^4 \to \sM_{15}$. Thus the Theorem follows if we establish the existence of such matrix factorizations.  We will prove the existence computationally in  Section \ref{constructions} with Theorem \ref{fam1}. \qed

\section{Betti Tables}\label{Betti Tables}

One approach to the desired matrix factorizations is via the study of the moduli space $\sM_X(7,c_1(\sF),c_2(\sF),c_3(\sF))$ of vector bundles on a general cubic 3-fold $X$. We choose a different more direct approach. 

Consider an $S$-modules $N$, annihilated by the equation of a cubic hypersurface $X$,  such that the 2-peroidic part of the minimal free the $S_X$-resolution gives the desired matrix factorization, or its transpose.
If we require in addition, that $S_X$-resolution derived from the minimal free $S$-resolution as in Section \ref{MatFac} is minimal right away, then up to twist there are only finitely many Betti tables possible. 

\begin{proposition} \label{tables}The Boij-S\"oderberg cone of $S$-Betti tables contains up to twist precisely 39 different integral tables $\beta^S(N)$ of projective dimension $\pd \beta^S(N) \le 4$ and $\codim \beta^S(N) \ge 3$, such that the induced possibly non-minmal $S_X$-resolution from Section \ref{MatFac} is minimal, and its 2-periodic part corresponds to  a matrix factorization of desired shape.  All of these tables satisfy  $\codim \beta^S(N) = 3$.
\end{proposition} 

\begin{proof} The possible  shape of the Betti table $\beta^S(N)$ is of the form
\begin{center}
\begin{tabular}{c|ccccc} 
  & 0 & 1 & 2 & 3 & 4 \\ \hline
0 & $a$ & . & . & . & . \\
1 & $b$ & $c$ & $d$ & . \\
2 & . & . & $e$ & $f$  & $h$ \\
3 & . & . & .  & . & $i$ \\
\end{tabular}
\quad or \quad
\begin{tabular}{c|ccccc} 
  & 0 & 1 & 2 & 3 & 4 \\ \hline
0 & $a$ & $b$ & . & . & . \\
1 & . & $c$ & $d$ & $e$ \\
2 & . & . & . & $f$  & $h$ \\
\end{tabular}
\end{center}
with $(a+d+h, b+e+i, c+f)=(3,15,18)$ or $(15,3,18)$ for the first shape, and 
$(a+d+h, b+e, c+f)=(18,15,3)$ or $(18,3,15)$ for the second shape.  Since all entries are nonnegative there are only finitely many tables to start with, and as a computation shows, $39$ of the tables lie in the Boij-S\"oderberg cone and satisfy $\codim \beta^S(N) \ge 3$. The last assertion follows by inspection of 
this list which we produced with \Mac \   using our package \href{http://www.math.uni-sb.de/ag/schreyer/home/computeralgebra.htm}{MatFac15}.\end{proof}

\begin{remark} The Picard group of a non-singular cubic is generated by the hyperplane class. This motivates $\codim \beta^S(N) \ge 3$ since otherwise
we have to guarantee that the class of the codimension 2 part of the support of $N$ is a multiple of the hyperplane class. The condition $\pd_S \beta^S(N) \le 4$ is motivated by the wish to  think of $N$ as a submodule of the global section module $\Gamma_*(\sL)$ of some auxiliary sheaf $\sL=\widetilde N$.
\end{remark}

\begin{example}\label{dE=11} The Betti tables of Proposition \ref{tables} with $\deg \beta^S(N) =11$ are the following:
\begin{center}
\begin{tabular}{c|ccccc} 
  & 0 & 1 & 2 & 3 \\ \hline
0 & 5 & 9 & . & .  \\
1 & . & 3 & 13 &6 \\
\end{tabular}
\quad and its dual \quad
\begin{tabular}{c|ccccc} 
  & 0 & 1 & 2 & 3  \\ \hline
1 & 6& 13 & 3 & . \\
2 & . & . &9 & 5  \\
\end{tabular} \; ,
\end{center}
\begin{center}
\begin{tabular}{c|ccccc} 
  & 0 & 1 & 2 & 3 & 4 \\ \hline
0 & 1 & . & . & . & . \\
1 & 3& 10 & 1 & . &.\\
2 & . & . & 12 & 8  & 1\\
\end{tabular}
\quad and \quad
\begin{tabular}{c|ccccc} 
  & 0 & 1 & 2 & 3 & 4 \\ \hline
0 & 6 & 12 & . & . & . \\
1 & . & . & 11 &3 & . \\
2 & . & . & . & 3  & 1 \\
\end{tabular}
\end{center}
\end{example}

\begin{example} The Betti tables of Proposition \ref{tables} with $\deg \beta^S(N) =13$ are the following:
\begin{center}
\begin{tabular}{c|ccccc} 
  & 0 & 1 & 2 & 3 \\ \hline
0 & 2 & . & . & .  \\
1 & 2& 15 & 13 & . \\
2 & . & . & 1 & 3  \\
\end{tabular}
\quad and its dual \quad
\begin{tabular}{c|ccccc} 
  & 0 & 1 & 2 & 3  \\ \hline
0 & 3 & 1 & . & .  \\
1 & . & 13& 15 &2  \\
2 & . & . & . & 2  \\
\end{tabular} \ .
\end{center}

\end{example}

\begin{example} \label{ACM} The Betti tables of Proposition \ref{tables} with $\pd \beta^S(N)=3$, i.e., ACM-tables, are the four ACM-tables above, the tables
\begin{center}
\begin{tabular}{c|ccccc} 
  & 0 & 1 & 2 & 3 \\ \hline
0 & 1 & . & . & .  \\
1 & 4& 12 & 2 & . \\
2 & . & . & 11& 6  \\
\end{tabular}
\qquad and \qquad
\begin{tabular}{c|ccccc} 
  & 0 & 1 & 2 & 3  \\ \hline
0 & 4 & 3 & . & .  \\
1 & . & 12& 14 & . \\
2 & . & . & . & 3  \\
\end{tabular} \, ,
\end{center}
\begin{center}
\begin{tabular}{c|ccccc} 
  & 0 & 1 & 2 & 3 \\ \hline
0 & 2 & . & . & .  \\
1 & 2& 11 & 1 & . \\
2 & . & . & 13 & 7 \\
\end{tabular}
\qquad and \qquad
\begin{tabular}{c|ccccc} 
  & 0 & 1 & 2 & 3  \\ \hline
0 & 3 & . & . & .  \\
1 & . & 10& . & .  \\
2 & . & . & 15& 8  \\
\end{tabular} 
\end{center}
of degree $\deg \beta^S(N)= 14,16,17$ and $20$ respectively, and their duals.
\end{example}

In Section \ref{constructions} we will construct unirational families of pairs $(N,X)$ of a module and a cubic 3-fold $X$, whose equation annihilates $N$, for many of the Betti tables above. Naturally we seek
for families which depend on at least $42$ parameters modulo projectivities. Our approach is the following:
The module $N$ will sheafify to a line bundle $\sL$ on an auxiliary (smooth and irreducible) curve $E$ of degree $d_E= \deg \beta^S(N)$. The geometric genus $g_E$ and the degree $\deg \sL$ of the line bundle
are not determined by $\beta^S(N)$. However, $h^0(\sL)$ and $h^1(\sL)$ are determined by   $\beta^S(N)$,  at least if we make some plausible assumptions  on the local cohomology module $$H^1_\gm(N) \cong \Gamma_*(\sL)/ N. $$ Here $\gm \subset S$ denotes the homogeneous maximal ideal. 
It is natural to assume $h^0(\sO_E(1))=5$. However, the speciality $h^1(\sO_E(1))$ is another undetermined quantity.
Our choice of $g_E, h^1(\sO_E(1))$ and $\deg \sL$ is motivated by a dimension count. We will construct generically reduced families of dimension $\ge 42$ for the following 
Betti tables.
$$
\begin{tabular}{ | l | l |}
\hline
$d_E=11$ & $d_E=14$ \cr \hline
\begin{tabular}{c|ccccc} 
  & 0 & 1 & 2 & 3 \\ \hline
0 & 5 & 9 & . & .  \\
1 & . & 3 & 13 &6 \\
\end{tabular}
& %\quad  \hbox{ and }\qquad &
\begin{tabular}{c|ccccc} 
  & 0 & 1 & 2 & 3 & 4\\ \hline
0 & 2 & . & . & .  & .\\
1 & 1& 9 & . & . & .\\
2 & . & . & 14& 9 & 1  \\
\end{tabular}\cr \hline

\begin{tabular}{c|ccccc} 
  & 0 & 1 & 2 & 3 & 4 \\ \hline
0 & 6 & 12 & . & . & . \\
1 & . & . & 11 &3 & . \\
2 & . & . & . & 3  & 1 \\
\end{tabular}
&% \quad \hbox{ and }  \qquad &
\begin{tabular}{c|ccccccc} 
  & 0 & 1 & 2 & 3  \\ \hline
0 & 6 & 11 & . & .  \\
1 & . &  2 & 12 &  4 \\
2 & . & .  &  .  & 1 \\
\end{tabular} \cr \hline

\begin{tabular}{c|ccccc} 
  & 0 & 1 & 2 & 3 & 4 \\ \hline
0 & 1 & . & . & . & . \\
1 & 3& 10 & 1 & . &.\\
2 & . & . & 12 & 8  & 1\\
\end{tabular} 

& %\quad \hbox{ and } \quad &
\begin{tabular}{c|ccccccc} 
  & 0 & 1 & 2 & 3  \\ \hline
0 & 7 & 15 & 4 & .  \\
1 & . &  . & 8 &  3 \\
2 & . & .  &  .  & 1 \\
\end{tabular}  \cr \hline
\end{tabular}
$$
The  tables in the first column are realized by modules $N$ such that the line bundle $\sL=\widetilde N$ has support  on a curve $E$ residual to a line in the cubic threefold. The tables
in the second column are realized  by modules on curves $E$ of degree $14$. The last table is an example not covered by Propostion \ref{tables}.

\section{Constructions}\label{constructions}

Let us call a matrix factorization $(\phi,\psi)$ of $18\times(15+3)$ and $(15+3)\times18$ matrices on a cubic threefold $X$ \emph{good}, if the complex as in Proposition \ref{monad} is a monad
of an ideal sheaf of a \emph{smooth} curve $C$ of degree $d_C=16$ and genus $g_C=15$ such that $\sO_C(1) \in W^4_{16}(C)$ is a \emph{smooth isolated point}, i.e., 
$(C,\sO_C(1)$ does not lie in the ramification loci of $\widetilde \sM_{15,16}^4 \to \sM_{15}$ . If $(\phi,\psi)$ is good, then we call $(\psi,\phi(-3)), (\psi^t, \phi^t), (\phi^t,\psi^t(-3))$ and twists of these matrix factorizations good as well. In our constructions below an auxiliary curve $E$, a line bundle $\sL$ on $E$, and a  submodule $N \subset \Gamma_*(\sL)$ play a role.
We will always denote by
$$ d_E, g_E, d_\sL$$
the degree $\deg \sO_E(1)$ of $E$, the  genus of  $E$ and the degree of $\sL$, respectively.

Perhaps the easiest case is the construction for a Betti table of type
\begin{center}
\begin{tabular}{c|ccccc} 
  & 0 & 1 & 2 & 3 \\ \hline
0 & 5 & 9 & . & .  \\
1 & . & 3 & 13 &6 \\
\end{tabular}
\end{center}
We have to choose  $(E, \sO_E(1))$ a curve together with a very ample line bundle $\sO_E(1)$ of degree $d_E=11$ with $h^0(\sO_E(1))=5$, a cubic form $f \in H^0(\sI_E(3))$ up to a scalar, and the line bundle $\sL$.
From Riemann-Roch,
$$
h^0(\sO_E(1))-h^1(\sO_E(1)) = d_E + 1 -g_E
$$ 
we obtain $h^1(\sO_E(1))= g_E-7$.
Hence we expect that the pair $(E, \sO_E(1))$ depends on
$$
4g_E-3-h^0(\sO_E(1))h^1(\sO_E(1))=32-g_E
$$
parameters. Assuming that cubics cut a complete non-special linear series on $E$, we obtain
$$
34-(3d_E+1-g_E)=g_E
$$
parameters for the choice of $X$. Finally, since $N$ is an ACM-module, we have
$N=\Gamma_*(\sL)$ and 
$$
h^1(\sL)=h^0(\omega_E \tensor \sL^{-1} ) = \dim \Ext^3_S(N,S(-5))_0 = \beta_{3,5}^S(N)=0.
$$
Thus, the line bundle is non-special of degree $\deg \sL= g_E-1+5$ by Riemann-Roch, and  depends on $g_E$ parameters. Altogether we have
$$
32+g_E
$$
parameters.

\begin{theorem} \label{fam1}There exists a $42$-dimensional unirational family of tuples
$$(E,\sO_E(1),X,\sL) \hbox{ with } (d_E,g_E,d_\sL)=(11,10,14)$$
of a smooth curve $E$, a very ample line bundle $\sO_E(1)$ of degree $d_E=11$ and $h^0(\sO_E(1))=5$, a smooth cubic hypersurface  $X \subset \PP(H^0(\sO_E(1))) \cong \PP^4$ containing the image of $E$, and non-special line bundles $\sL$ on $E$, such that $N=\Gamma_*(\sL)$ has an $S$-resolution with Betti table $\beta^{S}(N)$
\begin{center}
\begin{tabular}{c|ccccc} 
  & 0 & 1 & 2 & 3 \\ \hline
0 & 5 & 9 & . & .  \\
1 & . & 3 & 13 &6 \\
\end{tabular}
\end{center}
such that for a general tuple the $S_X$ resolution of $N$ gives a good matrix factorization of desired shape.
\end{theorem}

\begin{proof} Since $h^1(\sO_E(1))=3$ we expect that $E$ has a plane model of degree $2g_E-2-d_E=7$ and $\delta={6 \choose 2}-g_E=5$ double points.
So we start with 5+10 general points $p_1, \ldots,p_5,q_1,\ldots,q_{10} \in \PP^2$ and a curve
$E' \subset \PP^2$ of degree $7$ with double points in $p_1,\ldots,p_5$ and simple points in $q_1,\ldots,q_{10}$. Let $E$ be the normalization of $E'$, $\sO_E(1)=\omega_E(-H)$
where $H$ denotes a general hyperplane section of $E' \subset \PP^2$ and
$\sL= \omega_E(q_1+q_2+q_3-(q_4+\ldots+q_{10}))$. The complete linear system $|\sO_E(1)|$ is cut out by plane cubics through $p_1,\ldots,p_5$,  so re-embeds $E$ into $\PP^4$ as a curve on a Del Pezzo surface $Y$ of degree $4$.  Hence $H^0(Y, \sI_{E/Y}(3))  \cong H^0(\PP^2, \sI_{\{p_1, \ldots,p_5\}}(2))$ is one-dimensional and $h^0(\PP^4, \sI_E(3))= 2\cdot 5+1=g_E+1$
as desired. Note, that $E$ is residual to  a line in the complete intersection of the cubic with two quadrics. 
Counting parameters we find $2\cdot 15-8=22$ parameter for the choice of the points up to projectivities, ${7+2 \choose 2}-3\cdot 5- 10 -1 = 10$ parameter for the choice of $E$   and another $10$ for the choice of $X$. 
So altogether we get the desired $42$. Clearly our parameter space is unirational. 

To verify  for a general point in this unirational parameter space, that the curve $E$  and the cubic $X$ are smooth,
that the module $N=\Gamma_*(\sL)$ has syzygies as expected, that the matrix factorization leads to a smooth curve $C$ of degree $d_C=16$ and genus $g_C=15$ such that
$\sO_C(1) \in W^4_{16}(C)$ is a smooth isolated point, can be done by producing a single example with these properties, because these properties are open conditions.
It is even enough to check this in an example defined over a finite  prime field $\FF_p$, 
since we may regard such an example as the reduction modulo $p$ of an example defined over the integers. By semicontinuity
the example over the generic point of $\Spec \ZZ_{(p)}$ is an example defined over $\QQ$, which has all desired properties.
We pick our example over a moderate-size prime field at random, and check all assertions with the computer algebra system \Mac \ using  the package \href{http://www.math.uni-sb.de/ag/schreyer/home/computeralgebra.htm}{MatFac15}. This computation completes the proof of Theorem \ref{fam1} and also the proof of Theorem \ref{main}.\end{proof}

\begin{remark} The reader might wonder why I did not try to construct a family of such modules $N$ using a curve $E$ of genus $g_E=11$. 
In this case the line bundle $\sO_E(2)$ of degree $22$ is non-special, hence $E \subset \PP^4$ would lie on at least $15-(22+1-11)=3$ quadrics, which by B\'ezout must have a surface $Y$ in common of degree $\le 3$. Since $35-(33+1-11)=12 $ I expect that $h^1(\sI_E(3))=1$, $h^0(\sI_E(3))=13$,
and that there are 2 linear syzygies among the quadrics. in this case, the surface $Y$ would be a cubic scroll, and the equation of the cubic hypersurface $X$ would be a linear combination of the quadrics, hence the determinant of a $3\times 3$ linear matrix. So $X$ would be singular. Since a general curve $C \in \sH$ lies on smooth cubic 3-fold by Theorem \ref{fam1}, the deduced family of curves of genus $g_C=15$ cannot dominate $\sM_{15}$. 
\end{remark}

Next we discuss the third table of Example \ref{dE=11}. As before we set $\sL= \widetilde N$. Since $H^1_\gm(N)$ is dual to $\Ext_S^4(N,S(-5))$ it is reasonable to assume that
$H^1_\gm(N) \cong K(-2)$ so that the Betti tables $\beta^S(N)$ and $ \beta^S(\Gamma_*(\sL))$ differ by a Koszul complex on the 5 linear forms: 
\begin{center}
\begin{tabular}{c|ccccc} 
  & 0 & 1 & 2 & 3 & 4 \\ \hline
0 & 1 & . & . & . & . \\
1 & 3& 10 & 1 & . &.\\
2 & . & . & 12 & 8  & 1\\
\end{tabular}
\quad and \quad
\begin{tabular}{c|ccccc} 
  & 0 & 1 & 2 & 3 \\ \hline
0 & 1 & . & . & .  \\
1 & 4& 15 & 11 & . \\
2 & . & . &2 & 3   \\
\end{tabular}
\end{center}
This time $\sL$ is a line bundle with $h^0(\sL)=1$ and $h^1(\sL)=3$, and degree $\deg \sL = g_E-3$. So $\sL$ is determined by the choice of $g_E-3$ general points on $E$. On the other hand, the choice of $N \subset \Gamma_*(\sL)$ corresponds to choosing a $3$-dimensional subspace of the $4$-dimensional
space of generators of $\Gamma_*(\sL)$ in degree $1$, i.e., to a point in $\PP^3$. Thus the pair
$(N, \sL)$ depends, given $E \subset \PP^4$,  again on $g_E$ parameters. Thus, if we choose $E \subset X$ as in the proof of Theorem \ref{fam1} as a curve of genus $g_E=10$ residual to a line, we get again a $42$-dimensional family.

\begin{theorem}\label{fam2} There exists a $42$-dimensional unirational family of tuples
$$(E,\sO_E(1),X,\sL,N) \hbox{ with } (d_E,g_E,d_\sL)=(11,10,7)$$
of a smooth curve of genus $g_E=10$, a very ample line bundle $\sO_E(1)$ of degree $d_E=11$ and $h^0(\sO_E(1))=5$, a smooth cubic hypersurface  $X \subset \PP(H^0(\sO_E(1))) \cong \PP^4$ containing the image of $E$, an effective line bundle $\sL$ on $E$  and a submodule $N \subset \Gamma_*(\sL)$, such that $N$ has an $S$-resolution with Betti table 
\begin{center}
\begin{tabular}{c|ccccc} 
  & 0 & 1 & 2 & 3 & 4 \\ \hline
0 & 1 & . & . & . & . \\
1 & 3& 10 & 1 & . &.\\
2 & . & . & 12 & 8  & 1\\
\end{tabular}
\end{center}
such that for general tuples the $S_X$-resolution of $N$ gives a good matrix factorization of desired type.
\end{theorem}

\begin{proof} This time we take $E'$ as a septic with $5$ nodes, passing through additional $7$ simple points, which we use to define the line bundle $\sL$. The dimension count reads
\begin{center}
\begin{tabular}{rcr|l}
$2(5+7)-8$&$=$&$16$ & for the points in the plane up to projectivities \\
$36-3\cdot5-7-1$&$=$& $13$& for $E'$ \\
&&$3$ & for  $N \subset \Gamma_*(\sL)$ \\
&&$10$ & for the cubic $X \supset E$ \\ \hline
&& $42$ & parameters altogether \\
\end{tabular}
\end{center}
To check that general choices lead to a tuple with all desired properties follows again by a computation of a random example over a finite prime field.
See \href{http://www.math.uni-sb.de/ag/schreyer/home/computeralgebra.htm}{MatFac15}  for details. \end{proof}

For the last table of Example \ref{dE=11}, it is reasonable to assume that $\Ext_S^4(N,S(-6)) \cong S/(I_L+\gm^2)$, where $I_L$ denotes the homogeneous ideal of a line $L$. Then $H^1_\gm(N)$ has the Betti table
\begin{center}
\begin{tabular}{c|cccccc} 
  & 0 & 1 & 2 & 3 & 4 &5 \\ \hline
0 & 2 & 9 & 15 & 11  & 3 &.\\
1 & .  & .   & 1 &  3& 3 &1\\
\end{tabular}
\end{center}
Since  $\beta^S(N)$ and $\beta^S(\Gamma_*(\sL))$ differ by this table we get
\begin{center}
\begin{tabular}{c|ccccc} 
  & 0 & 1 & 2 & 3 & 4 \\ \hline
0 & 6 & 12 & . & . & . \\
1 & . & . & 11 &3 & . \\
2 & . & . & . & 3  & 1 \\
\end{tabular}
\quad and \quad
\begin{tabular}{c|ccccc} 
  & 0 & 1 & 2 & 3  \\ \hline
0 & 8 & 21 & 15 & . \\
1 & . & . & 1 &3  \\
\end{tabular}
\end{center}
for these tables. Note that we can recover the line $L$ from $\sL$: Its equations are given by the linear entries in the last syzygy matrix of $\Gamma_*(\sL)$.
So with $N'=\Hom_K(S/(I_L+\gm^2),K(-1))$ and $\phi \in \Hom_S(\Gamma_*(\sL),N')_0$ a surjective morphism, we can take
$N = \ker \phi$. In all cases computed, I found $\dim \Hom_S(\Gamma_*(\sL),N')_0 =1$, so that in these case $N$ is determined by $\sL$.
If this is true in general, then we can obtain another  42-dimensional unirational family as follows: Start again with $5+10$ points $p_1,\ldots,p_5,q_1,\ldots q_{10} \in \PP^2$ and  a general septic $E'$ with nodes in $p_1,\ldots,p_5$ and simple points in $q_1,\ldots q_{10}$, hence geometric  genus $g_E=10$.  As the non-special line bundle $\sL$ on $E$ of degree $\deg \sL=17$ we can take
$\sL = \sO_E(1) \tensor \sO_E(q_1+\ldots+q_8-q_9-q_{10})$. Then $\Gamma_*(\sL)$ determines a line $L$ and hence a module $N'$ as above. 

\begin{proposition} \label{fam3} If for general choices,
 $ \Hom_S(\Gamma_*(\sL),N')_0\not=0$ holds for  the construction above, then this gives  a $42$-dimensional unirational family of pairs $(N,X)$ with invariants
 $$(d_E,g_E,d_\sL)=(11,10,17)$$
  such that for general tuples the $S_X$ syzgyies give a good matrix factorization of desired shape.

\end{proposition}

\begin{proof} This is another computer algebra verification documented in \href{http://www.math.uni-sb.de/ag/schreyer/home/computeralgebra.htm}{MatFac15}.  \end{proof}

\begin{remark} Note that from our examples of Proposition \ref{fam3}, we can conclude $\dim \Hom_S(\Gamma_*(\sL),N')_0 \le 1$ for general $\Gamma_*(\sL)$ by semi-continuity. It is very unlikely that $ \Hom_S(\Gamma_*(\sL),N')_0=1$ does not for general choices, because this  would mean that our randomly chosen  examples, by accident, all lie in a proper subfamily. Having tested several examples over an field of approximate size $10^4$, this is nearly impossible. This is no rigorous proof, which might be actually be easy. I did not seriously tried to proof this, in view of Proposition \ref{tangentSpaces}. \end{remark}

Next we discuss a family which rises from modules with Betti table $\beta^S(N)$
\begin{center}
\begin{tabular}{c|ccccc} 
  & 0 & 1 & 2 & 3 & 4\\ \hline
0 & 2 & . & . & .  & .\\
1 & 1& 9 & . & . & .\\
2 & . & . & 14& 9 & 1  \\
\end{tabular}
\quad and \quad 
\begin{tabular}{c|ccccc} 
  & 0 & 1 & 2 & 3 \\ \hline
0 & 2 & . & . & .  \\
1 & 2& 14 & 10 & . \\
2 & . & . &  4& 4  \\
\end{tabular}
\end{center}
for  $\beta^S(\Gamma_*(\sL))$ which differ by a Koszul complex.
In this case $E$ has degree $d_E=14$ and assuming $h^0(\sO_E(1))=5$ we obtain $h^1(\sO_E(1))=g_E-10$, hence expect
$4g_E-3-h^0(\sO_C(1))\cdot h^1(\sO_C(1))=47-g_E$ parameters for the pair $(E, \sO_E(1))$.  The number of cubic hypersurfaces containing the image of $E$ in $\PP H^0(\sO_C(1)) \cong \PP^4$ is expected to be $g_E-9$. The line bundle $\sL = \widetilde N$ has degree   $\deg \sL = g_E-3$ and depends on $g_E -2\cdot 4 $ parameters. Finally, choosing
$N \subset \Gamma_*(\sL)$ corresponds to the choice of a point in $\PP^1$, which gives one more parameter. Altogether we have $g_E+31$ parameters. Choosing $g_E=11$ we can hope for a dominant family. The model of $E \subset \PP^3$ embedded by $\omega_E \tensor \sL^{-1}$ is the space model of degree 12 used by Chang and Ran to prove the unirationality of $\sM_{11}$. The line bundle $\sO_E(1) \cong \omega_E(-(p_1+\ldots +p_6)$ is the Brill-Noether dual to an effective divisor. We do not know how to construct $E$ together with $6$ points in a unirational way. 
But over a finite field one can easily find points in $E$ with a probabilistic method  Thus we are able to produce random elements in this family.

\begin{theorem}\label{fam4} There exists a $42$-dimensional family of tuples
$$(E,\sO_E(1),X,\sL,N) \hbox{ with } (d_E,g_E,d_\sL)=(14,11,8)$$
of a smooth curve $E$, a very ample line bundle $\sO_E(1)$ of degree $d_E=14$ and $h^0(\sO_E(1))=5$, a smooth cubic hypersurface  $X \subset \PP(H^0(\sO_E(1))) \cong \PP^4$ containing the image of $E$, a line bundle $\sL$ on $E$ such that $h^0(\sL)=2$  and a submodule $N \subset \Gamma_*(\sL)$, such that $N$ has an $S$-resolution with Betti table $\beta^{S}(N)$
\begin{center}
\begin{tabular}{c|ccccc} 
  & 0 & 1 & 2 & 3 & 4\\ \hline
0 & 2 & . & . & .  & .\\
1 & 1& 9 & . & . & .\\
2 & . & . & 14& 9 & 1  \\
\end{tabular}
\end{center}
such that the $S_X$ -resolution gives a good matrix factorization of desired type.
\end{theorem}

\begin{proof} This follows from another computation over a finite field documented in \href{http://www.math.uni-sb.de/ag/schreyer/home/computeralgebra.htm}{MatFac15}. \end{proof}

Our next construction is a family of modules $N$ with Betti table
\begin{center}
\begin{tabular}{c|ccccccc} 
  & 0 & 1 & 2 & 3  \\ \hline
0 & 6 & 11 & . & .  \\
1 & . &  2 & 12 &  4 \\
2 & . & .  &  .  & 1 \\
\end{tabular}
\end{center}
The support $E$ has degree $d_E=14$. We will construct $E$ as a  curve  residual to a rational normal curve $R$ of degree $d_R=4$ in a complete intersection $(2,3,3)$
of degree $18$. Such curves have a Betti table
\begin{center}
\begin{tabular}{c|ccccccc} 
  & 0 & 1 & 2 & 3  \\ \hline
0 & 1 & . & . & .  \\
1 & . &  1 & . &  . \\
2 & . &  2&  .  & . \\
3 & . &  3 & 10 & 5 \\
\end{tabular}
\end{center}
 and genus $g_E=15$. The line bundle $\sL= \widetilde N$ has to have degree $\deg \sL = 19$
 since $h^0(\sL)-h^1(\sL)=6-1=\deg \sL +1 -g_E.$
 So it has  the form $\sL=\omega_E(-D)$, where $D$ is an effective divisor of degree $9$ on $E$. A unirational construction of $N$ runs as follows:
 Start with a rational normal curve $R \subset \PP^4$ and choose one point $p_0$. Choose a quadric $Q$ containing $R$ and $p_0$. Choose $8$ lines $\ell_i$ through $p_0$
 and take $p_i$ as the second intersection point of $\ell_i\cap Q$. Then choose $X,X'$ two general cubic hypersurfaces through $R \cup \{p_0,\ldots,p_9\}$.
 The residual $E$ of $R$ in $Q \cap X \cap X'$ is the desired curve, and $D=p_0+\ldots+p_8$ is the desired effective divisor of degree $9$ on $E$

\begin{theorem} \label{fam5} Up to projectivities, there is a $47$-dimensional unirational family of tuples
$$(X,E,\sL) \hbox{ with } (d_E,g_E,d_\sL)=(14,15,19)$$
of a cubic hypersurface $X$, curves $E$ residual to a rational normal curve $R$ of degree $d_R=4$ in a complete intersection
$Q \cap X \cap X'$, where $Q$ is a quadric and $X'$ a further cubic hypersurface and $\sL=\omega_E(-D)$ for $D$ an effective divisor of degree $9$ on $E$ such that $N= H^0_*(\sL)$
has Betti table $\beta^S(N)$
\begin{center}
\begin{tabular}{c|ccccccc} 
  & 0 & 1 & 2 & 3  \\ \hline
0 & 6 & 11 & . & .  \\
1 & . &  2 & 12 &  4 \\
2 & . & .  &  .  & 1 \\
\end{tabular}
\end{center}
For general choices, the $S_X$ resolution gives a good matrix factorization of desired type. \end{theorem}

\begin{proof} Most of the result follows from a computation in \Mac \ documented in \href{http://www.math.uni-sb.de/ag/schreyer/home/computeralgebra.htm}{MatFac15}.  For the dimension count we note that the stabilizer of $R$ in $\PGL(5)$ has dimension $3$. Thus $R \cup \{p_0\}$ depends up to projectivities on one parameter, (the cross ratio if we think of $\PP^4=\PP(H^0(\PP^1,\sO(4)))$ as the linear system of quartic polynomials). Choosing $Q$ gives $4$ parameters since $h^0(\PP^4,\sI_{R\cup\{p_0\}}(2))=5$. The lines give $24=8*3$ parameters, and $X$
gives another $12=6\cdot 5-8-9-1=h^0(\sI_{R\cup\{p_0,\ldots,p_7\}}(3))-1$ parameters. Finally the choice of $X'$ are $12-5-1=6$ further parameters, since the construction depends on the equations of $X'$ only modulo the equation of $Q$ and $X$. Altogether this are
$1+4+24+12+6=47$ parameters. A tangent space computation at a general point shows that this space is generically smooth with its natural scheme structure.
\end{proof}

\begin{theorem} \label{fam5a} Up to projectivities, there is a $46$-dimensional family of tuples
$$(X,E,\sL) \hbox{ with } (d_E,g_E,d_\sL)=(14,14,18)$$
of a cubic hypersurface $X$, curves $E$ and  a line bundle $\sL=\omega_E(-D)$ for $D$ an effective divisor of degree $8$ on $E$ such that $N= H^0_*(\sL)$
has Betti table $\beta^S(N)$
\begin{center}
\begin{tabular}{c|ccccccc} 
  & 0 & 1 & 2 & 3  \\ \hline
0 & 6 & 11 & . & .  \\
1 & . &  2 & 12 &  4 \\
2 & . & .  &  .  & 1 \\
\end{tabular}
\end{center}
%6-1=18+1-14, 18+8=2*14-2
For general choices, the $S_X$ resolution gives a good matrix factorization of desired type.  \end{theorem}

\begin{proof} A maximal rank curve  $E$ has degree $14$ and genus $14$ we we have that $E$ lies on no quadric. However, since $h^1(E,\sO_E(1))=4$, these curves have a model
in $\PP^3$ and the corresponding component of the Hilbert scheme is unirational. The space of pairs $(E,X)$ up to projectivities is unirational of dimension $38$. The choice of $8$ points on $E$ gives further $8$ parameters. Checking an example by computation in \Mac \ documented in \href{http://www.math.uni-sb.de/ag/schreyer/home/computeralgebra.htm}{MatFac15}  implies the result.
\end{proof}

Our last two example concern the construction of  matrix factorizations from modules not covered by Proposition \ref{tables}. We will construct a unirational family of modules with Betti table
\begin{center}
\begin{tabular}{c|ccccccc} 
  & 0 & 1 & 2 & 3  \\ \hline
0 & 7 & 15 & 4 & .  \\
1 & . &  . & 8 &  3 \\
2 & . & .  &  .  & 1 \\
\end{tabular}
\end{center}
whose support is as in family of Theorem \ref{fam5} a curve $E$
which is residual to a rational normal curve $R$ of degree $4$ in a complete intersection $(2,3,3)$. This time the  line bundle $\sL=\widetilde N$ must have  degree $\deg \sL = 20 $
by Riemann-Roch. Thus $\sL = \omega_E(-D)$
where $D$ is an effective divisor of degree $8$. Following the same construction as for the family in Theorem \ref{fam5}, we get:

\begin{theorem} \label{fam6} Up to projectivities, there is a $46$-dimensional unirational family of tuples
$$(X,E,\sL) \hbox{ with } (d_E,g_E,d\sL)=(14,15,20)$$
of a cubic hypersurface $X$, curves $E$ residual to a rational normal curve $R$ of degree $d_R=4$ in a complete intersection
$Q \cap X \cap X'$, where $Q$ is a quadric and $X'$ a further cubic hypersurface and $\sL=\omega_E(-D)$ for $D$ an effective divisor of degree $8$ on $E$ such that $N= H^0_*(\sL)$
has Betti table $\beta^S(N)$
\begin{center}
\begin{tabular}{c|ccccccc} 
  & 0 & 1 & 2 & 3  \\ \hline
0 & 7 & 15 & 4 & .  \\
1 & . &  . & 8 &  3 \\
2 & . & .  &  .  & 1 \\
\end{tabular}.
\end{center}
% 7-1=20+1-15, 20+8=2*15-2
For general choices, the $S_X$ resolution gives a good matrix factorization of desired type. \end{theorem}

\begin{proof} Most of the results follow from a computation in \Mac \  documented in \href{http://www.math.uni-sb.de/ag/schreyer/home/computeralgebra.htm}{MatFac15}.  The dimension count gives 
$1+4+3\cdot 7+13+7=46$ parameters this time. A tangent space computation at a general point of the parameter space shows that this space is generically smooth with its natural scheme structure.
\end{proof}

\begin{theorem} \label{fam6a} Up to projectivities, there is $45$-dimensional  family of tuples
$$(X,E,\sL)\hbox{ with } (d_E,g_E,d_\sL)=(14,14,19)$$
of a cubic hypersurface $X$, curves $E$ residual to a rational normal curve $R$ of degree $d_R=4$ in a complete intersection
$Q \cap X \cap X'$, where $Q$ is a quadric and $X'$ a further cubic hypersurface and $\sL=\omega_E(-D)$ for $D$ an effective divisor of degree $7$ on $E$ such that $N= H^0_*(\sL)$
has Betti table $\beta^S(N)$
\begin{center}
\begin{tabular}{c|ccccccc} 
  & 0 & 1 & 2 & 3  \\ \hline
0 & 7 & 15 & 4 & .  \\
1 & . &  . & 8 &  3 \\
2 & . & .  &  .  & 1 \\
\end{tabular}.
\end{center}
%%7-1=19+1-14, 19+7=2*14-2
For general choices, the $S_X$ resolution gives a good matrix factorization of desired type. \end{theorem}

\begin{proof} This follows by the same strategy as for the family in Theorem \ref{fam5a}.
\end{proof}

\section{Tangent space computations}\label{tangComp}

Let $(N,X)$ be a pair of an auxiliary module $N$ in one of the examples of Section \ref{constructions} and a cubic hypersurface $X$ whose equation annihilates $N$. We would like to estimate the dimension of the family of $M$'s, and hence the dimension of the family of curves of genus $15$ obtained from our family of pairs $(N,X)$. 
For fixed $X$, the group $\Ext^1_{S_X}(M,M)_0$ is the space of infinitesimal homogeneous deformations of $M$ (or the matrix factorization). Since cubic threefolds depend on $10$-parameters up to projectivities, we expect that $\dim \Ext^1_{S_X}(M,M)_0 = 32$ for general choices of $(N,X)$.  However, the association 
\begin{align} \label{NtoM}
(N,X) \mapsto (M,X) 
\end{align}
might not be surjective.  Let $P = \ker(M \to N)$ be the kernel of the MCM approximation, which turned out to need no free summand in all cases. Then $P$ sits in a short exact sequence
$$
0 \to P \to M \to N \to 0, 
$$
and has finite projective dimension as an $S_X$-module.
We have a diagram

$$ 
\xymatrix{
\Ext^1_{S_X}(M,P) \ar[r] &        \Ext^1_{S_X}(M, M) \ar[r]             &   \Ext^1_{S_X}(M,N) \ar[r] & \Ext^2_{S_X}(M, P) \\
&   & \Ext^1_{S_x}(N,N) \ar[u] & \\
 &  & \Hom_{S_X}(P,N) \ar[u] &\\  }
% & & \Hom_{S_X} (M,N) \ar[u] &\\
% & & \Hom_{S_X} (N,N) \ar[u] &\\
% & & 0 \ar[u] \\}
$$
By the periodicity of the $S_X$-resolution of $M$, we have 
$$\Ext^i_{S_X}(M,P) \cong \Ext^{i+2}_{S_X}(M(3),P) \hbox{ for }i \ge 1.
$$
Since
$\Ext^i_{S_X}(M,S_X) = 0$ for $i\ge 1$, we obtain  $\Ext^i_{S_X}(M,P) = 0$ for $i\ge 1$, because $P$ has finite projective dimension.
Thus 
$$
\Ext^1_{S_X}(M, M) \cong   \Ext^1_{S_X}(M,N). 
$$
For our families constructed in Section \ref{constructions} in Theorems \ref{fam1}, \ref{fam2}, \ref{fam4} and Proposition \ref{fam3}, we expect $\dim \Ext^1_{S_X}(N,N)_0=32$ as the total family of $(N,X)$ depends on $42$ parameters, and cubics on $10$. As it turns out, $\dim \Ext^1_{S_X}(N,N)_0=32$ holds for examples from these 4 families.
So the map (\ref{NtoM}) is surjective on the level of tangent spaces if and only if the map
$$
\Hom_{S_X}(P,N)_0 \to \Ext^1_{S_X}(N,N)_0
$$ 
is zero. For the examples from the families Theorem \ref{fam5}, \ref{fam5a}, \ref{fam6}  and \ref{fam6a}, this map cannot be trivial, since $\dim \Ext^1_{S_X}(M,N)_0 =32$, while $\Ext^1_{S_X}(N,N)_0$
has dimension 37, 36, 36 and 35 respectively.
In all eight families of Section \ref{constructions} we  find that for general choices the map
$$
\Hom_{S_X}(N,N)_0  \to \Hom_{S_X}(M,N)_0
$$ 
is an isomorphism between 1-dimensional spaces. Hence there is only one map $M \to N$, and $\dim \Hom_{S_X}(P,N)_0$ is the dimension of the kernel
$$ 
\Ext^1_{S_X}(N,N)_0 \to \Ext^1_{S_X}(M,N)_0
$$ 

\begin{proposition} \label{tangentSpaces} $\dim \Hom_{S_X}(M,N)_0=1$ holds in a randomly chosen example over a moderate sized finite field in all eight families from Section \ref{constructions}. Moreover,
$$
\dim   \Hom_{S_X}(P,N)_0=
\begin{cases}
0 \\
3 \\
4\\
6 \\
7 \\
\end{cases} \hbox{ holds for some examples from }
\begin{cases}
\ref{fam4} \\
\ref{fam1}, \ref{fam3}, \ref{fam6a} \\
\ref{fam5a} \\
\ref{fam2}, \ref{fam5} \\
\ref{fam6} \\
\end{cases} \qquad \hfill
$$
and
$$
\dim \Ext^1_{S_X}(N,N)_0 =
\begin{cases}
32 \\
35 \\
36 \\
37 \\
\end{cases}
\hbox{ holds in an open set of the families}
\begin{cases}
\ref{fam1}, \ref{fam3}, \ref{fam4} \\
\ref{fam6a}\\ 
\ref{fam6}, \ref{fam5a} \\
\ref{fam5} \\
\end{cases} \hfill
$$
\end{proposition}

\noindent
\textit{Proof} by  computation documented in \href{http://www.math.uni-sb.de/ag/schreyer/home/computeralgebra.htm}{MatFac15}. The tangent space 
$$ 
\dim \Ext^1_{S_X}(N,N)_0
$$ cannot be smaller by the dimension count in our families.
\qed

\begin{corollary}\label{dim fam} The family from Theorem 
$$
\begin{cases}
\ref{fam1}, \ref{fam6} \\
\ref{fam5} \\
\ref{fam4}, \ref{fam5a} ,\ref{fam6a}
\end{cases}
\hbox{ maps onto  an (at least) }
\begin{cases}
39\\
41 \\
42 \\
\end{cases}
\hbox{-dimensional} \hfill
$$
subvariety 
 of $\sM_{15}$ ,respectively.
 The non-dominant families are unirational, the dominant family from Theorem \ref{fam4} is uniruled.
\end{corollary}

\begin{proof} In principle, it could be that the map from our families to $\sM_{15}$ is ramified at the given randomly chosen point. So the tangent space computation of Proposition \ref{tangentSpaces} gives only a lower bound for the dimension of the family in $\sM_{15}$. In reality it is very unlikely that equality does not hold. The last statement
follows from Theorem \ref{uniruled}.
\end{proof}

{\it Proof of Theorem \ref{uniruled} and Theorem \ref{OnGenCubic}.} The last choice in the construction of the dominant family from Theorem \ref{fam4} is the choice of a point in $\PP^1$. Thus $\widetilde \sM^4_{15,16}$ is ruled by lines. The same holds for a component of the Hilbert scheme of genus $15$ and degree $16$ curves on a general cubic threefold.
\qed

{\it Proof of Theorem \ref{random curve}.} We can use the dominant family from Theorem \ref{fam4}, Theorem \ref{fam5a} or Theorem \ref{fam6a}. For the family from Theorem \ref{fam4}, all computations but the choice of an effective divisor of degree $6$ on $E$  are computations with modules of fixed degree and regularity in a fixed number of variables. So these computations run in a  number of field operations in $\FF_q$ independent of the size of $q$. Hence they contribute with
$O((\log q)^2)$ to the running time. The computation of an effective divisor uses factorization of a univariate polynomial of fixed degree, hence  Berlekamp's algorithm, which runs in $O((\log q)^3)$ field operations in $\FF_q$, see, e.g., 
\cite{GG}, Theorem 14.14.
\qed

\begin{remark} The points in $\sM_{15}(\FF_q)$ which we can find with the probabilistic algorithm for elements from Theorem \ref{fam4} form only a fraction. If $B$ denotes the parameter space of the family, then both $B \to \sM^4_{15,16} \to \sM_{15}$ are generically finite. Thus, it is reasonable to expect  about $0.39\approx 0.63^2$ of the points in $\sM_{15}(\FF_q)$
to be in the image of $B(\FF_q)$ for reasonably large $q$. See \cite{EHS}, Section 2 for a discussion.
\end{remark}

\begin{remark} The construction of the family from Theorem \ref{fam5}, \ref{fam6} has some easy variants. Instead of starting with a rational normal curve  $R$ of degree $4$, we might start with some other curves $R$ of degree $4$ and take as degree $14$ curves $E$,  the residual in a complete intersection $(2,3,3)$. I inspected the cases, where  $R$ is
 \begin{enumerate}
\item an elliptic curve $R_1^4$ of degree 4, 
\item a rational normal curve $R^4_0$, the case of Theorem \ref{fam5} and \ref{fam6}
\item the disjoint union $R^1_3 \cup L$ of a plane elliptic curve and a line,
\item the disjoint union $R^3_0 \cup L$ of a twisted cubic and a line ,
\item the disjoint union  $C_1\cup C_2$ of two conics,
\item the disjoint union $C \cup L_1 \cup L_2$ of a conic and two lines,
\item the disjoint union $L_1 \cup \ldots \cup L_4$ of four lines.
 \end{enumerate}
None of these  12 further families dominates $\sM_{15}$. Indeed, except for the family from Theorem \ref{fam5}, which is 41-dimensional and its variant (4), which turned out to be 40-dimensional all of these unirational families are of dimension $\le 39$, with equality in 5 further cases, see \href{http://www.math.uni-sb.de/ag/schreyer/home/computeralgebra.htm}{MatFac15}. One natural explanation, why 
only the not obviously unirational families are dominant, and nearly all other constructions lead to 39-dimensional families, could be Conjecture \ref{mrc}. Of cause the evidence for this conjecture is rather weak.
 \end{remark}

\bigskip

\vbox{\noindent Author Address:\par
\smallskip
\noindent{Frank-Olaf Schreyer}\par
\noindent{Mathematik und Informatik, Universit\"at des Saarlandes, Campus E2 4, 
D-66123 Saarbr\"ucken, Germany}\par
\noindent{schreyer@math.uni-sb.de}\par
}

\end{document}